\documentclass[12pt]{article}
\usepackage{amsmath,amsthm,amssymb,tikz,subfigure, rotating}
\usepackage{fullpage}

\usetikzlibrary {positioning}

\definecolor {processblue}{cmyk}{0,0,0,1}

\newtheorem{theorem}{Theorem}[section]
\newtheorem{lemma}[theorem]{Lemma}
\newtheorem{corollary}[theorem]{Corollary}

\newtheorem{conjecture}[theorem]{Conjecture}

\DeclareMathOperator{\lcm}{lcm}

\begin{document} 
\title{Short cycles in digraphs and the Caccetta--H\"{a}ggkvist conjecture}
\author{Muhammad A. Khan\\
\small{Department of Mathematics \& Statistics, University of Calgary}\\ 
\small{2500 University Drive NW, Calgary AB, Canada, T2N 1N4.}\\
\small{\texttt{muhammkh@ucalgary.ca}}}

\date{}
\maketitle

\begin{abstract}
In the theory of digraphs, the study of cycles is a subject of great importance and has given birth to a number of deep questions such as the Behzad--Chartrand--Wall conjecture (1970) and its generalization, the Caccetta--H\"{a}ggkvist conjecture (1978). Despite a lot of interest and efforts, the progress on these remains slow and mostly restricted to the solution of some special cases. In this note, we prove these conjectures for digraphs with girth is at least as large as their minimum out-degree and without short even cycles. More generally, we prove that if a digraph has sufficiently large girth and does not contain closed walks of certain lengths, then the conjectures hold. The proof makes use of some of the known results on the Caccetta--H\"{a}ggkvist conjecture, properties of direct products of digraphs and a construction that multiplies the girth of a digraph.

\vspace{3mm}  


\vspace{3mm} 

\end{abstract}


\section{Introduction}\label{sec:intro}


Let $D$ be a digraph with set of vertices $V(D)$, set of arcs $A(D)$ and order $\left|V(D)\right|$. We write $g(D)$ for the (directed) girth and $\delta^+ (D)$ for the minimum out-degree of a vertex in $D$. Mostly, we simplify matters by representing the order of a digraph by $n$, its girth by $g$ and its minimum out-degree by $k$. A digraph is said to be {\it simple} if it contains no parallel arcs directed from one vertex to another. An {\it oriented graph} is a simple digraph without any pair of symmetric arcs. All digraphs considered in this paper are simple but not necessarily oriented (unless explicitly stated). Therefore, we can represent an arc directed from a vertex $u$ to a vertex $v$ of a digraph $D$ by $uv$. Moreover, here the terms `cycle' and `girth' always refer to a directed cycle and the directed girth, respectively. 

The investigation of the cycle structure of digraphs mostly focuses on particular kinds of cycles and systems of cycles in digraphs. The study of shortest cycles forms an important part of this investigation.  To this end, it is natural to look for good upper bounds on the girth of digraphs. A digraph is said to be {\it $d$-regular} (or {\it $d$-biregular}) if every vertex has in-degree and out-degree $d$. A $d$-regular digraph of girth $g$ and of least order is called a {\it directed $(d,g)$-cage}. In 1970, Behzad, Chartrand, and Wall \cite{BCW} conjectured that the order of a directed $(d,g)$-cage is $d(g-1)+1$. Equivalently, we have \cite{S2}: 

\begin{conjecture}[Behzad--Chartrand--Wall conjecture]\label{bcw}
Let $G$ be a $k$-regular digraph of order $n$ and girth $g$. Then $g\le \lceil n/k\rceil$.  
\end{conjecture}

The above conjecture was proved for $k=2$ by Behzad \cite{Beh}, for $k=3$ by Bermond \cite{Ber}, and for vertex-transitive digraphs by Hamidoune \cite{H1}. In 1978, Caccetta and H\"{a}ggkvist \cite{CH} proposed to replace $k$-regularity in the statement of Conjecture \ref{bcw} by specifying the minimum out-degree in the digraph instead. Thus the Caccetta--H\"{a}ggkvist conjecture was born, which is currently one of the central open questions in graph theory. 

\begin{conjecture}[Caccetta--H\"{a}ggkvist conjecture]\label{CH conjecture}
Every digraph of order $ n $ with minimum out-degree at least $ k $ has a cycle with length at most $ \lceil n/k\rceil $. 
\end{conjecture}

Only special cases of Conjecture \ref{CH conjecture} have been resolved. The case $k=2$ was solved in the original paper of Caccetta and H\"{a}ggkvist \cite{CH}. Later the conjecture was proved for $k=3$ by Hamidoune \cite{H2} and $k=4,5$ by Ho\`{a}ng and Reed \cite{HR}. Moreover, Shen \cite{S2} showed that the conjecture holds for all $k\leq \sqrt{n/2}$. Among large $k$, the case corresponding to $k=n/2$ trivially holds. However, the case $k=n/3$, or more generally $n/3\leq k< n/2$, is already open, highly interesting and, for digraphs with minimum in-degree and out-degree both at least $k$, is implied by Seymour's second neighbourhood conjecture \cite{Sul}. Some results on this case appear in \cite{Lich, Raz}. Another problem related to this case asks to determine the smallest $\alpha>0$ such that any digraph of order $n$ and minimum out-degree at least $\alpha n$ contains a cycle of length at most 3. If the Caccetta--H\"{a}ggkvist conjecture holds then $\alpha = 1/3$. However, this is still open. The upper bound on $\alpha$ was progressively improved in \cite{CH, Bo, S1, HHK}, in chronological order, to the current best of $\alpha \le 0.3465$ in \cite{Hal}.


Since cycles of length $\lceil n/k \rceil$ or less proved elusive, Chv\'{a}tal and Szemer\'{e}di \cite{CS} proposed to look for cycles of length at most $\frac{n}{k} + c$, for small $c$. They proved the existence of such cycles for $c = 2500$. This was improved to $c=304$ by Nishimura \cite{Nish} and to $c=73$ by Shen \cite{S3}. For small $n$, the best upper bound of  
$$g\le 3\left\lceil \frac{n}{k}\ \ln\left(\frac{2+\sqrt{7}}{3} \right)\right\rceil \approx 1.312 \ \frac{n}{k}$$
is due to Shen \cite{S3} that improves the earlier result $g\le 2n/(k+1)$ by Chv\'{a}tal and Szemer\'{e}di \cite{CS}. Shen \cite{S2} also proved the following result, which we use in the sequel and which implies that any counterexample to Conjecture \ref{CH conjecture} satisfies $n\le 2k^2 -3k$. 

\begin{theorem}\label{summary} 
Let $D$ be a digraph with $\left|V(D)\right|=n$, $\delta^{+}(D)=k$ and $g(D)=g$. Then 
\begin{equation}\label{eq:best1}
g\le \max\left\{\left\lceil\frac{n}{k}\right\rceil, 2k-2\right\}.
\end{equation}
\end{theorem}

In January 2006, Chudnovsky, Seymour and Thomas organized a workshop at the American Institute of Mathematics to investigate the Caccetta--H\"{a}ggkvist conjecture and its relatives. Sullivan \cite{Sul} wrote an extensive survey of the different forms of the conjecture and related problems discussed at the workshop. However, no major progress has been made since Shen's papers \cite{S2, S3}.   

Here we prove the Caccetta--H\"{a}ggkvist conjecture (and hence the Behzad--Chartrand--Wall conjecture) for any digraph $D$ with odd girth $g(D)\ge \delta^+ (D)$ and not containing even cycles of length less than $2g(D)$. The proof relies on some properties of direct product of digraphs and a construction that doubles the girth of a digraph with odd girth. The direct product of digraphs and other ingredients of our approach are discussed in the next section, while the proof appears in section \ref{sec:ch}. The same method shows that the Caccetta--H\"{a}ggkvist conjecture is satisfied by any digraph $D$ of girth $g(D)\ge 2\delta^+ (D)/p$ not containing closed walks of lengths $p, 2p, \ldots, (g(D)-1)p$, for some positive integer $p\ge 2$. This is a draft version of the paper and will be expanded soon.

\section{Walks, adjacency and direct products}\label{sec:direct}
Recall that all digraphs considered in this paper are simple. A {\it walk of length $\ell$} in a digraph $D$ is a sequence $v_0 , v_1 , \ldots, v_\ell$ of not necessarily distinct vertices of $D$ so that for each $i=1,\ldots, \ell$, $v_{i-1}v_i \in A(D)$. If $v_0 = v_\ell$, we have a {\it closed walk} \cite{Bang}. Moreover, a closed walk on distinct vertices is a cycle. Clearly, a walk can have repeated arcs that form a multiset, which we call the {\it arc multiset} of the walk. We need the following fundamental result concerning closed walks in a digraph. A slightly weaker form of Lemma \ref{decomp} appears in \cite[Exercise 1.12]{Bang}, whereas the form we use appears in \cite{Trim}.    

\begin{lemma}\label{decomp}
The arc multiset of a closed walk in a digraph decomposes into arc sets of (not necessarily distinct) cycles. Thus the length of a closed walk in a digraph equals the sum of lengths of the (not necessarily distinct) cycles it traverses.  
\end{lemma}

The {\it adjacency matrix} $M_D$ of a digraph $D$ of order $n$ is an $n\times n$ matrix whose $ij$-entry, $(M_D )_{ij}$, is $1$ if $v_i v_j \in A(D)$ and $0$ otherwise \cite{Bang}. The matrix $M_D$ can be used to compute the number of closed walks of length $\ell$ containing a given vertex of $D$ in the following way. 

\begin{lemma}\label{adj}
Let $D$ be a digraph with vertex set $V(D)=\{v_1 , \ldots, v_n \}$ and corresponding adjacency matrix $M_D$. Then $(M_D ^\ell )_{ii}$, the $ii$-entry of the matrix power $M_D^\ell$, equals the number of closed walks of length $\ell$ containing the vertex $v_i \in V(D)$. 
\end{lemma}

A (digraph) {\it product} is a binary operation defined on the class of all digraphs such that given digraphs $D_1$ and $D_2$, their product is a digraph with vertex set $V(D_1 )\times V(D_2 )$ -- the Cartesian product of vertex sets of $D_1$ and $D_2$ -- and whose arcs are defined according to some condition depending on the arcs of $D_1$ and $D_2$. Some of the popular product operations include Cartesian product, direct product, strong product and lexicographic product (see \cite[Chapter 32]{Hand} for details). Here we are interested in the direct product. 

The direct product of graphs was explicitly defined in \cite{We} and extended to digraphs in a natural way in \cite{Mc}. We denote the direct product of digraphs $D_1$ and $D_2$ by $D_1 \times D_2$ and define it to be a digraph with vertex set $V(D_1 ) \times V(D_2 )$ such that an arc is directed from a vertex $(u,v)$ to a vertex $(x,y)$ if and only if $ux \in A(D_1 )$ and $vy \in A(D_2 )$. This product operation, also known as the tensor product, Kronecker product or categorical product, has the distinction of being the category-theoretic product arising in the category of digraphs and homomorphisms \cite{Hand}. In addition, it has a natural connection with the Kronecker product of matrices that we describe a bit later. 

Given a $p\times q$ matrix $A$ and an $r\times s$ matrix $B$, the {\it Kronecker product} $A\otimes B$ is the $pr\times qs$ block matrix 
\[A\otimes B = \begin{bmatrix} a_{11}B & \cdots & a_{1q}B \\
a_{21}B & \cdots & a_{2q}B \\
\vdots & \ddots & \vdots \\
a_{p1}B & \cdots & a_{pq}B \\
\end{bmatrix}.
\]

The following lemma (see, for example, Lemma 4.2.10 in \cite{Horn}) relates Kronecker product with the usual matrix product. 

\begin{lemma}\label{mixed}
If $A$, $B$, $C$ and $D$ be matrices such that the products $AC$ and $BD$ are defined, then 
\[(A\otimes B)(C\otimes D) = AC\otimes BD. 
\]
\end{lemma}

Now we can describe the relationship between direct products of digraphs and Kronecker products of matrices. 

\begin{lemma}\label{product}
For any digraphs $D_1$ and $D_2$, 
\[M_{D_1 \times D_2 } = M_{D_1 } \otimes M_{D_2 }. 
\] 
\end{lemma}

Let $\ell$ be a positive integers. We denote by $C_\ell$ the digraph consisting of a directed cycle of length $\ell$. Also for any positive integer $j$ and a digraph $D$, let $jD$ denote the digraph consisting of the union of $j$ (vertex and arc) disjoint copies of $D$. The result below appears as relation (32.1) in \cite{Hand}. 
 
\begin{lemma}\label{cycles}
Given positive integers $\ell$ and $m$, 
\[C_\ell \times C_m = \gcd(\ell, m) \ C_{\lcm(\ell, m)}, 
\]
where $\gcd(\cdot , \cdot)$ and $\lcm(\cdot , \cdot)$ denote the greatest common factor and the least common multiple operators, respectively.  
\end{lemma}

\section{Constructing digraphs of twice the girth} \label{sec:ch}
In this section, we present our main result. As mentioned earlier, the proof technique make use of Theorem \ref{summary}, properties of direct product of digraphs and a construction of digraphs of larger girth from digraphs of smaller girth.   

\begin{theorem}\label{ch1}
Let $D$ be a digraph satisfying $g(D) \ge \delta^+ (D)$ and not containing any even cycle of length less than $2g(D)$, Then $D$ satisfies the Caccetta--H\"{a}ggkvist conjecture.  

\end{theorem}

\begin{proof}
Let $\left|V(D)\right|=n$, $\delta^+ (D) = k$ and $g(D)=g\ge k$. Let $D^\times = C_2 \times D$. Clearly, $V(D^\times ) = 2n$ and $\delta^+ (D^\times ) = k$ and $g$ is odd. We show that $g(D^\times ) = 2g$. By Lemma \ref{cycles}, $g(D^\times ) \le \lcm(2, g) = 2g$. We prove that $D^\times$ does not contain a cycle of length less than $2g$. By Lemma \ref{product}, 
\[M_{D^\times } = M_{C_2 \times D} = M_{C_2} \otimes M_{D} 
\] 
and using Lemma \ref{mixed}, 
\[M^\ell_{D^\times } = (M_{C_2} \otimes M_{D})^\ell = M^\ell_{C_2} \otimes M^\ell_{D},
\]
for any positive integer $\ell$. Now Lemma \ref{adj} implies that if $\ell$ is a positive integer for which $D^\times$ contains a closed walk of length $\ell$, then for some $i=1,\ldots, 2n$, $p=1, 2$ and $q=1,\ldots, n$, we have $(M^\ell_{D^\times })_{ii} = (M^\ell_{C_2})_{pp} (M^\ell_{D})_{qq}>0$. This occurs if and only if both $C_2$ and $D$ contain a closed walk of length $\ell$. Now by Lemma \ref{decomp}, $C_2$ contains closed walks of every even length and no closed walks of odd length, whereas the closed walk of smallest even length in $D$ has length $2g$. Thus $\ell\ge 2g$, which gives $g(D^\times ) \ge 2g$. Thus $g(D^\times ) = 2g$, as claimed. 

Since $g(D^\times ) = 2g\ge 2k = 2\delta^+ (D^\times )$, by Theorem \ref{summary},   
\[2\left\lceil\frac{n}{k}\right\rceil \ge \left\lceil\frac{2n}{k}\right\rceil = \left\lceil\frac{\left|V(D^\times )\right|}{\delta^+ (D^\times )}\right\rceil \ge g(D^\times ) = 2g.  
\]
This gives   
\[\left\lceil\frac{n}{k}\right\rceil \ge g,  
\]
completing the proof.
\end{proof}

One can generalize the technique used in the proof of Theorem \ref{ch1} to obtain the following result. 

\begin{corollary}\label{ch2}
Let $p> 2$ be a positive integer and $D$ be a digraph with $g(D)\ge 2\delta^+ (D)/p$ not containing closed walks of lengths $p, 2p, \ldots, (g(D)-1)p$. Then $D$ satisfies the Caccetta--H\"{a}ggkvist conjecture. 
\end{corollary}

\begin{proof}
Let $g$, $k$ and $n$ be as in the proof of Theorem \ref{ch1}. Consider $D^\times_p = C_p \times D$ and note that $V(D^\times_p ) = pn$ and $\delta^+ (D^\times_p ) = k$. By suitably generalizing the arguments in the proof of Theorem \ref{ch1}, we see that $g(D^\times_p ) = pg$. 

Since $g(D^\times_p ) = pg\ge 2k = 2\delta^+ (D^\times_p )$, by Theorem \ref{summary},   
\[p\left\lceil\frac{n}{k}\right\rceil \ge \left\lceil\frac{pn}{k}\right\rceil = \left\lceil\frac{\left|V(D^\times )\right|}{\delta^+ (D^\times )}\right\rceil \ge g(D^\times ) = pg,   
\]
yielding   
\[\left\lceil\frac{n}{k}\right\rceil \ge g,   
\]
as desired. 
\end{proof}

The key difference between Theorem \ref{ch1} and Corollary \ref{ch2} is that for the former it suffices to consider digraphs without short even cycles, whereas for the latter one has to rule out the existence of closed walks of lengths that are multiples of $p$. On the other hand, the conclusion of Theorem \ref{ch1} is valid only if the girth of the digraph is at least as large as its minimum out-degree, while Corollary \ref{ch2} is applicable to digraphs of relatively smaller girths. 

The next result is an immediate consequence of Theorem \ref{ch1} and Corollary \ref{ch2}. 

\begin{corollary}
Conjecture \ref{bcw} holds for all $k$-regular digraphs having (i) girth $g\ge k$ that do not contain any even cycle of length less than $2g$ and (ii) girth $g\ge 2k/p$ that do not contain closed walks of lengths $p, 2p, \ldots, (g-1)p$, for some positive integer $p> 2$. 
\end{corollary}

\section*{Acknowledgments} The author is supported by a Vanier Canada Graduate Scholarship (NSERC) and an Alberta Innovates Technology Futures (AITF) award.


\end{document}